\documentclass[11pt, letterpaper]{amsart}
\usepackage{amsfonts}
\usepackage{amsmath}
\usepackage{amssymb}
\usepackage{enumerate}
\usepackage{tikz}
\usepackage{mathrsfs}
\usepackage{mathtools}
\usepackage{amsthm}
\usepackage{CJK}

\newtheorem{prop}{Proposition}
\newtheorem{cor}{Corollary}
\newtheorem{lem}{Lemma}

\newtheorem*{thm1}{Theorem 1}
\newtheorem*{thm2}{Theorem 2}
\newtheorem{conj}{Conjecture}

\begin{CJK}{UTF8}{}
\gdef\san{\mbox{\textbf{山}}}

\end{CJK}

\newcommand{\Ok}{\mathcal{O}_K}
\DeclareMathOperator{\ord}{ord}

\begin{document}
\begin{CJK}{UTF8}{min}
\title[Uniform Finiteness for CM Elliptic Curves]{A Uniform Version of a Finiteness Conjecture for CM Elliptic Curves}
\author{Abbey Bourdon}

\begin{abstract}
Let $A$ be an abelian variety defined over a number field $F$. For a prime number $\ell$, we consider the field extension of $F$ generated by the $\ell$-powered torsion points of $A$. According to a conjecture made by Rasmussen and Tamagawa, if we require these fields to be both a pro-$\ell$ extension of $F(\mu_{\ell^{\infty}})$ and unramified away from $\ell$, examples are quite rare. Indeed, it is expected that for a fixed dimension and field of definition, there exists such an abelian variety for only a finite number of primes.

We prove a uniform version of the conjecture in the case where the abelian varieties are elliptic curves with complex multiplication. In addition, we provide explicit bounds in cases where the number field has degree less than or equal to 100.
\end{abstract}

\maketitle

\section{Introduction}

Galois representations are an indispensable tool for analyzing the structure of the absolute Galois group of a number field $F$. One example of considerable interest comes from the fact that $G_F \vcentcolon=\text{Gal}(\bar{F}/F)$ acts (up to inner automorphism) on the algebraic fundamental group of the projective line minus three points. More precisely, if we let $X=\mathbb{P}^1_{F} \setminus \{0,1,\infty\}$ and $\bar{X}=X \otimes_F \bar{F}$, we may relate $\pi_1(X)$ and $\pi_1(\bar{X})$ through the following exact sequence:
\[
1 \to \pi_1(\bar{X}) \to \pi_1(X) \to G_F \to 1.
\]

\noindent From this we deduce the natural representation
\[
\Phi \colon G_F \to \text{Out}(\pi_1(\bar{X})).
\]

One approach to studying this representation, championed by Ihara in the 1980s, is to fix a prime number $\ell$ and consider the corresponding representation involving the pro-$\ell$ fundamental group of $\bar{X}$. Since $\pi_1^{\ell}(\bar{X})$ is a characteristic quotient of $\pi_1(\bar{X})$, we may define $\Phi_{\ell} \colon G_F \to \text{Out}(\pi_1^{\ell}(\bar{X}))$ via the following commutative diagram:

\begin{center}
\begin{tikzpicture}[node distance=2cm]
 \node (G)  {$G_F$};
 \node (1) [right of=G, node distance = 2.4 cm] {$\text{Out}(\pi_1(\bar{X}))$};
 \node (2)  [below of=1] {$\text{Out}(\pi_1^{\ell}(\bar{X}))$};

 \draw[->] (G) edge node [auto] {$\Phi$} (1);
 \draw[->] (G) edge node [left] {$\Phi_{\ell}$} (2);
 \draw[->] (1) to node {} (2);
 
\end{tikzpicture}
\end{center}

\noindent The fixed field of the kernel of $\Phi_{\ell}$, which we will denote $\san(F, \ell)$, is a large subfield of $\bar{F}$ that as of yet is not entirely understood. We know from Anderson and Ihara's 1988 paper \cite{ihara} that $\san(F, \ell)$ is a pro-$\ell$ extension of $F(\mu_{\ell^{\infty}})$ unramified away from $\ell$, but it is still an open question, first posed by Ihara in the case where $F=\mathbb{Q}$, to determine whether  $\san(F, \ell)$ is the \emph{maximal} such extension.

Subfields arising from geometric objects have helped to shed light on the structure of  $\san(F, \ell)$. For example, by the work of Rasmussen, Papanikolas, and Tamagawa in \cite{ras04}, \cite{ras07}, and \cite{ras1}, we know that if $E$ is an elliptic curve defined over $\mathbb{Q}$ with complex multiplication by $\mathbb{Q}(\sqrt{-\ell})$ and good reduction away from $\ell$, the field $\mathbb{Q}(E[\ell^\infty])$ is a subfield of  $\san(\mathbb{Q}, \ell)$. However, a finiteness conjecture made by Rasmussen and Tamagawa in \cite{ras1} implies such examples arising from abelian varieties are quite rare. As this conjecture motivates our work, we pause to introduce some notation and formalize its statement.

Let  $\mathscr{A}(F,g,\ell)$ be the set of $F$-isomorphism classes of abelian varieties $A/F$ of dimension $g$ for which $F(A[\ell^\infty])$ is both a pro-$\ell$ extension of $F(\mu_{\ell^{\infty}})$ and unramified away from $\ell$. We will use $\mathscr{A}(F, g)$ to denote the disjoint union of the $\mathscr{A}(F,g,\ell)$ over the set of all primes $\ell$, i.e.,  $\mathscr{A}(F, g)\vcentcolon=\{([A], \ell): [A] \in \mathscr{A}(F,g,\ell) \}$. Then we may state the Rasmussen-Tamagawa finiteness conjecture as follows:
\vspace{1 mm}
\begin{conj}
Let $F$ be a number field and $g>0$. The set $\mathscr{A}(F, g)$ is finite.
\end{conj}

\noindent This implies that, for a fixed $F$ and $g$, there exists a constant $C$ for which $\mathscr{A}(F,g,\ell)=\varnothing$ when $\ell>C$. One may ask whether stronger behavior should be expected for the bound $C$, and indeed the following uniform (meaning uniform in the degree of $F/\mathbb{Q}$) conjecture appears in \cite[Conj. 2]{ras2}:
\vspace{1 mm}
\begin{conj}
Let $F$ be a number field and $g>0$. There exists a constant $C$ depending only on $g$ and the degree of $F/ \mathbb{Q}$ for which $\mathscr{A}(F,g,\ell)=\varnothing$ when $\ell>C$.
\end{conj}

In \cite{ras1}, Rasmussen and Tamagawa prove Conjecture 1 in the cases where $g=1$ and $[F:\mathbb{Q}]=1$ or $2$, excluding the 9 imaginary quadratic fields of class number one. In addition, they find a complete list of $\mathbb{Q}$-isogeny classes in $\mathscr{A}(\mathbb{Q}, 1)$, showing that $\mathscr{A}(\mathbb{Q},1,\ell)$ is empty if $\ell > 163$. Later work by Ozeki in \cite{ozeki} proves Conjecture 1 for abelian varieties with complex multiplication defined over a number field containing the CM field, and Arai has explored the conjecture in the context of QM-abelian surfaces (see \cite{arai}).

More recently in \cite{ras2}, Rasmussen and Tamagawa prove that the Generalized Riemann Hypothesis implies Conjecture 1, and they give unconditional proofs in several new cases where the degree of $F$ is restricted and $g \leq 3$. (In particular, they resolve the conjecture in the case where $F$ is an imaginary quadratic field of class number one and $g=1$.) In addition, they prove a slightly stronger version of Conjecture 2 under the Generalized Riemann Hypothesis for any $g$ and any $F$ of odd degree.

Despite this progress, Conjecture 2 is known unconditionally only in the case where $g=1$ and $[F:\mathbb{Q}]=1$ or $3$ (see \cite{ras1}, \cite{ras2}). Moreover, aside from the case when $F=\mathbb{Q}$, any known bounds are likely far from optimal. In this article, we prove a result stronger than Conjecture 2 for elliptic curves with complex multiplication, and we give improved bounds for number fields of degree $1 < n \leq 100$. Specifically, we have the following theorem:

\begin{thm1}
Let $F$ be a number field with $[F:\mathbb{Q}]=n$. There exists a constant $C=C(n)$ depending only on $n$ with the following property: If there exists a CM elliptic curve $E/F$ with $F(E[\ell^\infty])$ a pro-$\ell$ extension of $F(\mu_{\ell})$ for some rational prime $\ell$, then $\ell \leq C$.
\end{thm1}

We record the consequences of this theorem for Conjectures 1 and 2. Let $\mathscr{A}^{\mathrm{CM}}(F,g,\ell)$ be the subset of $\mathscr{A}(F,g,\ell)$ consisting of abelian varieties with complex multiplication, and define $\mathscr{A}^{\mathrm{CM}}(F,g)\vcentcolon=\{([A], \ell): [A] \in \mathscr{A}^{\mathrm{CM}}(F,g,\ell) \}$. Then as a direct consequence of Theorem 1, we have the following corollaries:

\begin{cor}
Let $F$ be a number field with $[F:\mathbb{Q}]=n$. There exists a constant $C=C(n)$ depending only on $n$ such that $\mathscr{A}^{\mathrm{CM}}(F,1,\ell) = \varnothing$ if $\ell > C$.
\end{cor}

\begin{cor}
$\mathscr{A}^{\mathrm{CM}}(F,1)$ is finite.
\end{cor}

Note that the bound of Theorem 1 is achieved even as we relax the ramification requirement, thereby allowing the inclusion of elliptic curves with bad reduction at primes other than $\ell$. However, without the ramification requirement, we cannot guarantee (nor should we expect) a finiteness result as in Corollary 2.

A discussion of computed bounds is included at the end of section 3. 

\subsection*{Notation}
\begin{itemize}

\item $\mu_{\ell}$ denotes the group of $\ell$th roots of unity in $\bar{\mathbb{Q}}$, and $\mu_{\ell^{\infty}}= \cup_{n\geq 1}\mu_{\ell^n}$.
\item For an abelian variety $A$ defined over a field $F$, we denote the extension of $F$ generated by the $\ell$-torsion points of $A$ by $F(A[\ell])$. The field $F(A[\ell^\infty])$ is generated over $F$ by the $\ell$-powered torsion points of $A$.
\item If $F$ is a number field, $d_F$ is the absolute discriminant of $F/ \mathbb{Q}$. We denote the ring of integers of $F$ by $\mathcal{O}_F$, and $\mathcal{O}_F^{\times}$ is its group of units. 
\item $w_F$ denotes the number of distinct roots of unity in $F$, i.e., $|\mathcal{O}_F^{\times}|= w_F$.
\item If $\mathfrak{a}$ is an integral ideal in the number field $F$, we denote the norm of $\mathfrak{a}$ by $\mathcal{N}(\mathfrak{a})$. In other words, $\mathcal{N}(\mathfrak{a})= [\mathcal{O}_F:\mathfrak{a}]$.
\item $ \left( \frac{a}{\ell} \right)$ is the Kronecker symbol.

\end{itemize}

\section{Background on Ray Class Fields and Elliptic Curves}

We first recall the definition of the $\mathfrak{m}$-ray class group. Though the theory exists in more generality, here we restrict our attention to the case where $K$ is an imaginary quadratic field so we may take $\mathfrak{m}$ to be an integral ideal of $\Ok$. Then relative to $\mathfrak{m}= \prod \mathfrak{p}^{m(\mathfrak{p})}$, we define the following two subsets:\\

\hspace{1 cm} $I_K(\mathfrak{m})=$ the set of all fractional ideals of $K$ relatively prime to $\mathfrak{m},$\\

\hspace{1 cm} $P_K(\mathfrak{m}) = \{ \left(\alpha \right): \alpha \in K^{\times}, \ord_{\mathfrak{p}}(\alpha -1) \geq m(\mathfrak{p}) \text{ for all } \mathfrak{p} \text{ dividing } \mathfrak{m} \}.$\\

\noindent Note $P_K(\mathfrak{m})$ is a subgroup of $I_K(\mathfrak{m})$, and the quotient $I_K(\mathfrak{m})/P_K(\mathfrak{m})$ is the $\mathfrak{m}$-ray class group of $K$. As with the ideal class group, whose definition we recover when $\mathfrak{m}=1$, the $\mathfrak{m}$-ray class group is finite. In fact, we have the following explicit formula for its cardinality:

\begin{prop} Let $\mathfrak{m}$ be an integral ideal in a number field $K$. The order of the ray class group modulo $\mathfrak{m}$ is given by:
\[
h_{\mathfrak{m}}=h_K \cdot [U:U_{\mathfrak{m}}]^{-1} \cdot \mathcal{N}(\mathfrak{m}) \cdot \prod_{\mathfrak{p} \mid \mathfrak{m}} \left( 1- \mathcal{N}(\mathfrak{p})^{-1} \right)
\]
where
\begin{align*}
h_K &= \text{class number of } K\\
U &= \Ok^{\times}\\
U_{\mathfrak{m}} &= \{ \alpha \in U :  \ord_{\mathfrak{p}}(\alpha -1) \geq m(\mathfrak{p}) \text{ for all } \mathfrak{p} \text{ dividing } \mathfrak{m} \}.
\end{align*}

\end{prop}

\begin{proof}
See Corollary 3.2.4 in \cite{cohen2}. Note we have restricted to the case where our modulus is an integral ideal.
\end{proof}

\noindent In the special case $\mathfrak{m}=\ell \Ok$, we obtain:

\begin{cor}
Let $\ell$ be a prime and $\mathfrak{m}$ be the modulus $\ell \Ok$ in a quadratic field $K$. Then:

\begin{enumerate}
\item If $\ell$ is ramified in $K$, $h_{\mathfrak{m}}=h_K \cdot [U:U_{\mathfrak{m}}]^{-1} \cdot \ell \cdot (\ell -1).$
\item If $\ell$ splits in $K$, $h_{\mathfrak{m}}=h_K \cdot [U:U_{\mathfrak{m}}]^{-1} \cdot (\ell -1) \cdot (\ell -1).$
\item If $\ell$ is inert in $K$, $h_{\mathfrak{m}}=h_K \cdot [U:U_{\mathfrak{m}}]^{-1} \cdot (\ell +1) \cdot (\ell -1).$
\end{enumerate}
\end{cor}

Just as the ideal class group gives the Galois group of the Hilbert class field, in general the $\mathfrak{m}$-ray class group gives the Galois group of an abelian extension of $K$ called the ray class field of $K$ with modulus $\mathfrak{m}$, denoted by $K_{\mathfrak{m}}$. Ramification in the extension $K_{\mathfrak{m}}/K$ is restricted to primes dividing $\mathfrak{m}$, and the primes that split completely are precisely the primes in $P_K(\mathfrak{m})$. The fact that a unique $K_{\mathfrak{m}}$ exists for any ideal $\mathfrak{m}$ of $\Ok$ is a consequence of the Existence Theorem of Class Field Theory, and we direct the interested reader to Chapter 5 of \cite{janusz} for details.

We can construct the ray class fields of an imaginary quadratic field $K$ using torsion points of elliptic curves possessing complex multiplication. Recall that if $E/F$ is an elliptic curve, we say $E$ has complex multiplication, or CM, if its ring of $\bar{F}$-endomorphisms is strictly larger than $\mathbb{Z}$. In this case, $\text{End} (E)\otimes \mathbb{Q}$ is isomorphic to an imaginary quadratic field $K$, and $\text{End} (E)$ is isomorphic to an order in that field. If $E$ has CM by the maximal order in $K$, then the ray class field of $K$ with modulus $N\Ok$ can be generated from the $N$-torsion points of $E$, as we will now explain.

Since char($F) \neq 2 \text{ or } 3$, $E$ is isomorphic over $F$ to a curve having an equation of the form $y^2=4x^3-g_2x-g_3$, with $g_2, g_3 \in F$. The Weber function $\mathfrak{h}$ on $E$ is defined as follows, where $j_E$ is the $j$-invariant of $E$ and $\Delta=g_2^3-27g_3^2$:
\[
  \mathfrak{h}(x,y) = \left\{
  \begin{array}{l l}
    \vspace{3 mm}
    \dfrac{g_2g_3}{\Delta}x \text{ if } j_E \neq 0, 1728,\\
    \vspace{3 mm}
    \dfrac{g_2^2}{\Delta}x^2 \text{ if } j_E=1728,\\
    \dfrac{g_3}{\Delta}x^3 \text{ if } j_E=0.\\
  \end{array} \right.
\]

\noindent We then obtain an explicit description of the ray class field from the following theorem, the roots of which can be traced back to the work of Hasse in ~\cite{hasse}:

\begin{thm2}
Let $E$ be an elliptic curve defined over $K(j_E)$ with End($E$) $\cong \mathcal{O}_K$ for some imaginary quadratic field $K$. Let $\mathfrak{h}$ be the Weber function on $E$. Then $K(j_E, \mathfrak{h}(E[N]))$ is the ray class field of $K$ with modulus $N\Ok$.
\end{thm2}

\begin{proof}
See, for example, Theorem 2 in ~\cite[p.126]{lang}.
\end{proof}

Note that the Weber function is model independent. That is, if $\varphi \colon E \to E'$ is an $\bar{F}$-isomorphism and $\mathfrak{h}_E$, $\mathfrak{h}_{E'}$ the Weber functions of $E$ and $E'$, respectively, we have $\mathfrak{h}_E=\mathfrak{h}_{E'} \circ \varphi$. (See \cite[p.107]{shimura}.) This allows us to extend the result of Theorem 2 to include elliptic curves defined over an arbitrary number field $F$.

\begin{cor}
Let $E$ be an elliptic curve defined over a number field $F$ with End($E$) $\cong \mathcal{O}_K$ for some imaginary quadratic field $K$. Then $K(j_E, \mathfrak{h}(E[N]))$ is the ray class field of $K$ with modulus $N\Ok$.
\end{cor}

\begin{proof}
Let $E$ be an elliptic curve defined over a number field $F$ with End($E$) $\cong \mathcal{O}_K$. $E$ is isomorphic over $\mathbb{C}$ to an elliptic curve $E'$ defined over $K(j_E)$. (See \cite[p.105]{advanced}.) If we let $\varphi$ denote the isomorphism from $E$ to $E'$, the model independence of the Weber function gives $\mathfrak{h}_E(E[N])=\mathfrak{h}_{E'}(\varphi(E[N]))=\mathfrak{h}_{E'}(E'[N]).$ By Theorem 2, $K(j_{E'}, \mathfrak{h}_{E'}(E'[N]))=K(j_{E}, \mathfrak{h}_E(E[N]))$ is the ray class field of $K$ modulo $N\Ok$, as desired.
\end{proof}

\section{Proof of Main Result}

For an arbitrary elliptic curve, the mod-$\ell$ Galois representation is an injective homomorphism Gal($F(E[\ell])/F) \to \text{GL}_2(\mathbb{F}_{\ell})$. As a consequence, $[F(E[\ell]):F]$ must divide \#$\text{GL}_2(\mathbb{F}_{\ell})$. However, if $E$ has CM, much more is known:

\begin{prop}
Let $E$ be an elliptic curve with CM by $\Ok$ in $K$. Then for an odd prime $\ell$:
\begin{enumerate}
\item $\text{If } \left( \frac{d_K}{\ell} \right)=1,$ then $[F(E[\ell]):F] \mid 2(\ell -1)^2.$
\item $\text{If } \left( \frac{d_K}{\ell} \right)=-1,$ then $[F(E[\ell]):F] \mid 2(\ell^2 -1).$
\item $\text{If } \left( \frac{d_K}{\ell} \right)=0,$ then $[F(E[\ell]):F] \mid 2(\ell^2 -\ell).$
\end{enumerate}
\end{prop}

\begin{proof}
See, for example, Corollary 17 of ~\cite{clark1}.
\end{proof}

From these conditions, we see that it is only possible for an odd prime $\ell$ to divide the degree of $F(E[\ell])/F$ when $\ell$ divides $d_K$. This is a fact we are able to exploit:

\begin{lem} Suppose $E$ is an elliptic curve defined over a number field $F$ with complex multiplication by $\Ok$ in $K$. Suppose $\ell$ is prime and $\ell >\dfrac{w_K}{2}n+1$, where $n$ is the degree of $F/ \mathbb{Q}$. If $[F(E[\ell]): F(\mu_{\ell})]$ is $\ell$-powered, $\ell$ must divide $d_K$.
\end{lem}

\begin{proof}
Let $[F(E[\ell]):F(\mu_{\ell})]$ be $\ell$-powered, and suppose $\ell$ is an odd prime which does not divide $d_K$. Since $\ell \neq 2$, Proposition 2 forces $F(E[\ell])=F(\mu_{\ell})$. We will show this is a contradiction unless $\ell \leq \dfrac{w_K}{2}n+1$.

By Lemma 15 in  ~\cite{clark1}, $K \subset F(E[\ell])$. Thus $F(E[\ell])$ also contains $K(j_E, \mathfrak{h}(E[\ell]))$, the ray class field of $K$ modulo $\mathfrak{m}=\ell \Ok$. This gives us the following diagram of fields:

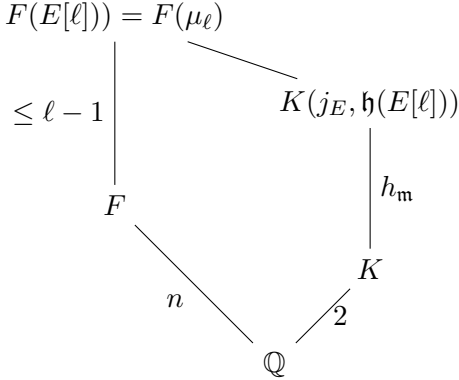
\begin{center}
\begin{tikzpicture}[node distance=2cm]
 \node (Q)                  {$\mathbb{Q}$};
 \node (F)  [above left of=Q, node distance=3 cm]   {$F$};
 \node (K) [above right of=Q, node distance=1.8cm] {$K$};
 \node (Fm) [above of =F, node distance=2.5cm] {$F(E[\ell]))=F(\mu_{\ell})$};
 \node (Kh) [above of =K, node distance=2.2 cm] {$K(j_E, \mathfrak{h}(E[\ell]))$};

 \draw[-] (Q) edge node[auto] {$n$} (F);
 \draw[-] (Q) edge node[right] {2} (K);
 \draw[-] (F) edge node[auto] {$\leq \ell-1$} (Fm);
 \draw[-] (K) edge node[right] {$h_{\mathfrak{m}}$} (Kh);
 \draw (Kh) -- (Fm);
 
\end{tikzpicture}
\end{center}

From Corollary 3, if $\ell$ splits in $K$ then
\begin{align*}
h_{\mathfrak{m}} &= h_K \cdot [U:U_{\mathfrak{m}}]^{-1} \cdot (\ell -1) \cdot (\ell -1)\\
&\geq 1 \cdot \dfrac{1}{w_K}\cdot (\ell -1) \cdot (\ell -1).\\
\end{align*}
Similarly, if $\ell$ is inert in $K$, 
\begin{align*}
h_{\mathfrak{m}} \geq 1 \cdot \dfrac{1}{w_K}\cdot (\ell +1) \cdot (\ell -1).\\
\end{align*}
In either case, $h_{\mathfrak{m}} \geq \dfrac{1}{w_K}\cdot (\ell -1)^2$. But
\[
2 \cdot \dfrac{1}{w_K}\cdot (\ell -1)^2 > n \cdot (\ell -1)
\]
whenever $\ell > \dfrac{w_K}{2}n+1$. In other words, $F(E[\ell]) \neq F(\mu_{\ell})$ for $\ell >  \dfrac{w_K}{2}n+1$. 
\end{proof}

In fact, the same result holds for elliptic curves with CM by an arbitrary order:

\begin{prop}
Suppose $E$ is an elliptic curve defined over a number field $F$ with complex multiplication by an order in $K$. Suppose $\ell$ is prime and $\ell >\dfrac{w_K}{2}n+1$, where $n$ is the degree of $F/ \mathbb{Q}$. If $[F(E[\ell]): F(\mu_{\ell})]$ is $\ell$-powered, $\ell$ must divide $d_K$.
\end{prop}

\begin{proof}
Suppose E has CM by the order $\mathcal{O}_f = \mathbb{Z} + f\Ok$ in $K$. Then there exists an $F$-rational isogeny $\varphi \colon E \to E'$ where $E'$ is defined over $F$ and has CM by $\Ok$. Since $\varphi$ is cyclic of degree $f$ (Proposition 25 in ~\cite{clark1}), we need only to show that $f$ and $\ell$ are relatively prime. This will ensure the induced map $\varphi \colon E[\ell] \to E'[\ell]$ is in fact an isomorphism, and since the isomorphism is defined over $F$, we will have $F(E[\ell])=F(E'[\ell])$. The result will then be a consequence of the previous lemma.

Let $f=p_1^{a_1} \cdots p_r^{a_r}$ be the prime factorization of $f$, with $p_1<p_2<\ldots<p_r$. As shown in ~\cite[p.146]{cox}, the class number of $\mathcal{O}_f$ satisfies:
\[
h(\mathcal{O}_f) = \frac{h(\Ok) p_1^{a_1 -1} \cdots p_r^{a_r -1}}{[\Ok^{\times}:\mathcal{O}_f^{\times}]} \prod_{i=1}^r \left(p_i-\left( \frac{d_K}{p_i} \right) \right).
\]
Since $| \Ok^{\times}| = w_K$ and $|\mathcal{O}_f^{\times}| \geq 2$,
\begin{align*}
h(\mathcal{O}_f) &\geq \frac{h(\Ok) p_1^{a_1 -1} \cdots p_r^{a_r -1}}{w_K/2} \prod_{i=1}^r \left(p_i-\left( \frac{d_K}{p_i} \right) \right)\\
&\geq \frac{2}{w_K} \left(p_r-\left( \frac{d_K}{p_r} \right) \right)\\
&\geq \frac{2}{w_K} \left(p_r-1 \right).
\end{align*}
We may obtain an upper bound on $h(\mathcal{O}_f)$ by recalling that $K(j_E)$ is the ring class field of $K$ of the order $\mathcal{O}_f$ (see ~\cite[p.220]{cox}). Thus $[K(j_E):K] =$ \#cl$(\mathcal{O}_f)$, and $h(\mathcal{O}_f) \leq n$. Combining this with the inequality above, we find $p_r \leq \dfrac{w_K}{2}n+1$. Since $\ell > \dfrac{w_K}{2}n+1$, this is enough to conclude $\ell$ and $f$ are relatively prime, as desired.
\end{proof}

If $n$ is odd we can extend the result to all odd primes:

\begin{cor}
Suppose $E$ is an elliptic curve defined over a number field $F$ with complex multiplication by an order in $K$. Suppose the degree of $F/ \mathbb{Q}$ is odd, and let $\ell$ be an odd prime number. If $[F(E[\ell]): F(\mu_{\ell})]$ is $\ell$-powered, $\ell$ must divide $d_K$.
\end{cor}

\begin{proof}
Suppose $\ell \nmid d_K$, and assume for the sake of contradiction that $[F(E[\ell]): F(\mu_{\ell})]$ is $\ell$-powered. By Lemma 15 in  ~\cite{clark1}, $K=\mathbb{Q}(\sqrt{D}) \subset F(E[\ell])$. In fact, $K \subseteq F(\mu_{\ell})$, for otherwise $F(\mu_{\ell})(\sqrt{D})$ would be a proper extension of $F(\mu_{\ell})$ contained in $F(E[\ell])$ and 2 would divide $[F(E[\ell]): F(\mu_{\ell})]$. However, since $\ell \nmid d_K$, we know $K \nsubseteq \mathbb{Q}(\mu_{\ell})$. Thus $\mathbb{Q}(\mu_{\ell})(\sqrt{D})$ is a proper extension of $\mathbb{Q}(\mu_{\ell})$ contained in $F(\mu_{\ell})$. Since $[F(\mu_{\ell}): \mathbb{Q}(\mu_{\ell})]=[F:\mathbb{Q}(\mu_{\ell}) \cap F]$, this forces $2 \mid [F:\mathbb{Q}(\mu_{\ell}) \cap F]$, which is a contradiction.
\end{proof}

We are now ready to prove our main result:

\begin{thm1}
Let $F$ be a number field with $[F:\mathbb{Q}]=n$. There exists a constant $C=C(n)$ depending only on $n$ with the following property: If there exists a CM elliptic curve $E/F$ with $F(E[\ell^\infty])$ a pro-$\ell$ extension of $F(\mu_{\ell})$ for some rational prime $\ell$, then $\ell \leq C$.
\end{thm1}

\begin{proof}
Suppose there exists a CM-elliptic curve  $E/F$ with $F(E[\ell^\infty])$ a pro-$\ell$ extension of $F(\mu_{\ell})$. Thus $[F(E[\ell]): F(\mu_{\ell})]$ is $\ell$-powered. Since $w_K \leq 6$ for an imaginary quadratic field $K$, the previous proposition shows $\ell \leq 3n+1$ or $\ell \mid d_K$ where $K$ is the CM-field of $E$. However, since $h(\Ok)$ divides \#cl$(\mathcal{O}_f)=[K(j_E):K] \leq n$, it follows that $K$ has class number less than or equal to $n$. As there are only a finite number of such $K$, proved by Heilbronn in \cite{heilbronn}, the result follows.
\end{proof}

It is clear that obtaining an explicit bound depends only on knowing the imaginary quadratic fields with a given class number, i.e., it depends on having a solution to the Gauss class number problem for imaginary quadratic fields. For class numbers up through 7 and odd class numbers up to 23, complete lists of the corresponding fields exist (see ~\cite{arno2} for a history of the many mathematicians involved in the early work on this problem and for a list of imaginary quadratic fields with odd class number up to 23; see ~\cite{stark}, ~\cite{arno}, ~\cite{wagner} for lists of imaginary quadratic fields of class number 2, 4, and 6, respectively). More recent work by Watkins in \cite{watkins} gives a solution for class numbers up to 100. To illustrate how these results may be used, we have compiled a table of bounds for elliptic curves defined over a number field $F$ of degree $n$ where $n \leq 7$:

\begin{center}
\begin{tabular}{l l}
$n$ & $\hfill C(n)$\\
\hline
1, 2 & \hspace{.5 cm} \hfill 163\\
3, 4 & \hfill 907\\
5, 6 & \hfill 2683\\
7 & \hfill 5923\\
\end{tabular}
\end{center}

We justify the claimed bounds. Suppose there exists a CM-elliptic curve $E/F$ with $F(E[\ell^\infty])$ a pro-$\ell$ extension of $F(\mu_{\ell})$. Then $[F(E[\ell]): F(\mu_{\ell})]$ is $\ell$-powered, and by Proposition 3, we know $\ell \leq \dfrac{w_K}{2}n+1 \leq 3n+1$ or $\ell$ divides $d_K$ where $K$ is the CM field of $E$. As mentioned in the proof of Theorem 1, the class number of $K$ is less than or equal to $n$, so we need only consult the lists of the discriminants of imaginary quadratic fields satisfying this constraint. A check of the possible primes dividing those discriminants yields the bounds above. Since Rasmussen and Tamagawa in  \cite{ras1} find an example of a CM-elliptic curve defined over $\mathbb{Q}$ with $\mathbb{Q}(E[163^{\infty}])$ a pro-$163$ extension of $\mathbb{Q}(\mu_{163})$, we see that in fact 163 is the best possible bound for $n=1$ and $n=2$.

We may also achieve a rough bound when $F$ has degree up to 100. In Table 4 from Watkins paper \cite{watkins}, he records the largest fundamental discriminant (in absolute value) for each class number up to 100. This is sufficient to generate additional bounds. For example, we know the largest fundamental discriminant  in Watkins's table, 2383747, occurs when $K$ has class number 98. Hence the largest possible prime dividing any discriminant is 2383739, so $C(n) \leq 2383739$ for all $n \leq 100$.

\section{Closing Remarks}

Although finding the conditions necessary for $[F(E[\ell]): F(\mu_{\ell})]$ to be $\ell$-powered was enough to prove the uniform bound in this paper, it is desirable to discover sufficient conditions as well. As discussed in the introduction, elliptic curves possessing this characteristic help us better understand $\san(F, \ell)$, provided they also have good reduction away from $\ell$. Here, we present two additional applications.

If $[F(E[\ell]): F(\mu_{\ell})]$ is $\ell$-powered, this partially determines the form of the Galois representation attached to $E$. Let $\rho_{E,\ell}:G_F \to GL_2(\mathbb{F}_{\ell})$ be the mod $\ell$ Galois representation, let $\chi$ be the cyclotomic character mod $\ell$, and let $\delta = [\mathbb{F}_{\ell}^{\times}: \chi(G_F)]$. Then by a result of Rasmussen and Tamagawa:

\begin{lem}
Suppose $E$ is an elliptic curve defined over a number field $F$ where $[F(E[\ell]): F(\mu_{\ell})]$ is $\ell$-powered. Then there exists a basis of $E[\ell]$ with respect to which
\[
\rho_{E,\ell}(G_F)=
\left[\begin{matrix}
\chi^{i_1} & *\\ 0& \chi^{i_2}
\end{matrix}\right].
\]
Furthermore, $i_1$, $i_2$ may be chosen to be nonnegative integers less than $\frac{\ell - 1}{\delta}$.
\end{lem}

\begin{proof}
A version of this appears as Lemma 3 in \cite{ras1}, and a more general version appears  in \cite{ras2}. Note that although the result in \cite{ras2} is stated for abelian varieties $A/F$ where $F(A[\ell^\infty])$ is both a pro-$\ell$ extension of $F(\mu_{\ell})$ and unramified away from $\ell$, the ramification requirement is not used in the proof.
\end{proof}

Knowing the sufficient conditions for $[F(E[\ell]): F(\mu_{\ell})]$ to be $\ell$-powered would also establish when there exist $\ell$-torsion points rational over $F(\mu_{\ell})$:

\begin{lem}
Suppose $E$ is an elliptic curve defined over a number field $F$. $E$ has a non-trivial $\ell$-torsion point rational over $F(\mu_{\ell})$ if and only if $[F(E[\ell]): F(\mu_{\ell})]$ is $\ell$-powered.
\end{lem}

\begin{proof}
Suppose $P \in E[\ell]$ is rational over $F(\mu_{\ell})$. Then we can choose a basis $\{P, Q\}$ of $E[\ell]$, yielding Gal$(F(E[\ell])/F(\mu_{\ell})) \cong $
$\left<
\left[\begin{smallmatrix}
1&b\\ 0&1
\end{smallmatrix}\right]
\right>$,
where $b \in \mathbb{F}_{\ell}$. (Recall the determinant of this matrix will equal the cyclotomic character, which is trivial in this extension.) But this group has size 1 or $\ell$. The other direction is a result of the Orbit-Stabilizer Theorem, but we can also see it as an immediate consequence of Lemma 2.
\end{proof}

Unfortunately, the converse of Proposition 3 does not hold. As a counterexample, consider the elliptic curve defined by the equation
\[
y^2=x^3-595x+5586
\]
at the prime $\ell=7$. This curve has CM by an order in $K=\mathbb{Q}(\sqrt{-7})$, so $\ell$ divides $d_K$ and $\ell > 3\cdot1+1$. By Lemma 4 in \cite{Dieulefait}, $\sqrt{7}$ is contained in $\mathbb{Q}(E[\ell])$. Since $\sqrt{7}$ is not contained in $\mathbb{Q}(\mu_{\ell})$, we find that $\mathbb{Q}(\mu_{\ell})(\sqrt{7})$ gives a degree 2 extension of $\mathbb{Q}(\mu_{\ell})$ inside of $\mathbb{Q}(E[\ell])$. In other words, 2 divides $[F(E[\ell]): F(\mu_{\ell})]$ and so the desired extension is not $\ell$-powered.

\section{Acknowledgements}
The author is grateful to her advisor, Chris Rasmussen, for suggesting the problem and for his guidance in preparing this paper. The author also thanks Akio Tamagawa for his helpful comment on the Weber function.
\vspace{2 cm}

\bibliographystyle{amsplain}
\bibliography{bibliography1}

\end{CJK}
\end{document}